\documentclass[11pt]{amsart} 
\usepackage{amssymb}
\usepackage{verbatim}
\usepackage{hyperref}
\usepackage{color}

\usepackage{tikz}

\begin{document}
\setcounter{tocdepth}{1}

\newtheorem{theorem}{Theorem}    
\newtheorem{proposition}[theorem]{Proposition}
\newtheorem{corollary}[theorem]{Corollary}
\newtheorem{lemma}[theorem]{Lemma}
\newtheorem{sublemma}[theorem]{Sublemma}
\newtheorem{conjecture}[theorem]{Conjecture}
\newtheorem{claim}[theorem]{Claim}
\newtheorem{fact}[theorem]{Fact}
\newtheorem{observation}[theorem]{Observation}

\newtheorem{definition}{Definition}
\newtheorem{notation}[definition]{Notation}
\newtheorem{remark}[definition]{Remark}
\newtheorem{question}[definition]{Question}
\newtheorem{questions}[definition]{Questions}
\newtheorem{hypothesis}[definition]{Hypothesis}

\newtheorem{example}{Example}
\newtheorem{problem}[definition]{Problem}
\newtheorem{exercise}[definition]{Exercise}


\def\repair{\medskip\hrule\hrule\medskip}

\def\bff{\mathbf f}
\def\bE{\mathbf E}
\def\bF{\mathbf F}
\def\bK{\mathbf K}
\def\bP{\mathbf P}
\def\bx{\mathbf x}
\def\bi{\mathbf i}
\def\bk{\mathbf k}
\def\bt{\mathbf t}
\def\bc{\mathbf c}
\def\ba{\mathbf a}
\def\bw{\mathbf w}
\def\bh{\mathbf h}
\def\bn{\mathbf n}
\def\bg{\mathbf g}
\def\bc{\mathbf c}
\def\bs{\mathbf s}
\def\bp{\mathbf p}
\def\by{\mathbf y}
\def\bv{\mathbf v}
\def\be{\mathbf e}
\def\bu{\mathbf u}
\def\bm{\mathbf m}
\def\bxi{{\mathbf \xi}}
\def\bR{\mathbf R}
\def\by{\mathbf y}
\def\bz{\mathbf z}
\def\bfb{\mathbf b}
\def\bPhi{{\mathbf\Phi}}

\newcommand{\norm}[1]{ \|  #1 \|}
\newcommand{\Norm}[1]{ \big\|  #1 \big\|}

\def\scriptl{{\mathcal L}}
\def\scriptc{{\mathcal C}}
\def\scriptd{{\mathcal D}}
\def\scrapd{{\mathcal D}}
\def\scripts{{\mathcal S}}
\def\scriptq{{\mathcal Q}}
\def\scriptt{{\mathcal T}}
\def\scriptf{{\mathcal F}}
\def\scriptm{{\mathcal M}}
\def\scripti{{\mathcal I}}
\def\scriptr{{\mathcal R}}
\def\scriptb{{\mathcal B}}
\def\scripte{{\mathcal E}}
\def\scripta{{\mathcal A}}
\def\scriptn{{\mathcal N}}
\def\scriptv{{\mathcal V}}
\def\scriptz{{\mathcal Z}}
\def\scriptj{{\mathcal J}}
\def\scriptk{{\mathcal K}}
\def\scriptg{{\mathcal G}}
\def\scripth{{\mathcal H}}

\def\bk{\mathbf k}
\def\kernel{\operatorname{kernel}}
\def\spann{\operatorname{Span}}
\def\eps{\varepsilon}

\def\reals{\mathbb R}
\def\naturals{\mathbb N}
\def\integers{\mathbb Z}
\def\rationals{\mathbb Q}
\def\one{\mathbf 1}
\def\complex{{\mathbb C}\/}
\def\torus{{\mathbb T}\/}

\def\lt{{L^2}}

\def\three{\mathbf 3}
\def\four{\mathbf 4}

\def\bart{\bar t}
\def\barz{\bar z}
\def\barx{\bar x}
\def\bary{\bar y}
\def\barz{{\bar z}}
\def\bars{\bar s}
\def\barc{\bar c}
\def\baru{\bar u}
\def\barr{\bar r}

\def\distance{\operatorname{distance}}
\def\md{{\mathcal D}}

\def\lsharp{\Lambda^\sharp}
\def\lnatural{\Lambda^\natural}
\def\sS{{\mathbb S}}
\def\barsS{\overline{\mathbb S}}

\title{Examples of H\"older-stable Phase Retrieval}

 \author{Michael Christ}
\address{
        Michael Christ\\
        Department of Mathematics\\
        University of California \\
        Berkeley, CA 94720-3840, USA}
\email{mchrist@berkeley.edu}

 \author{Ben Pineau}
\address{
        Ben Pineau\\
        Department of Mathematics\\
        University of California \\
        Berkeley, CA 94720-3840, USA}
\email{bpineau@berkeley.edu}

 \author{Mitchell A.~Taylor}
\address{
        Mitchell A.~Taylor\\
        Department of Mathematics\\
        University of California \\
        Berkeley, CA 94720-3840, USA}
\email{mitchelltaylor@berkeley.edu}

\date{April 29, 2022. Revised January 31, 2023}

\begin{abstract}
Examples are constructed of infinite-dimensional subspaces 
$V\subset \lt(\mu)$ with the property that for any $f,g\in V$,
if $|f|$ is approximately equal to $|g|$ with respect to
the $\lt$ norm, then there exists a unimodular scalar $z$
such that $f$ is approximately equal to $zg$. 
\end{abstract}

\thanks{Research of the first author was supported in part by NSF grant DMS-1901413}

 \maketitle

Let $(X,\scripta,\mu)$ be a measure space. 
Let $V$ be a closed subspace of the (real or complex) Hilbert space $\lt = \lt(\mu)$.
Calderbank, Daubechies, Freeman, and Freeman
\cite{calderbank_etal} have studied 
subspaces of real-valued $\lt$ for which there exists $C<\infty$ satisfying
\begin{equation} \label{SPR}
\min\big(\norm{f-g}_{\lt},\,\norm{f+g}_\lt\big) \le C\big\|\,|f|-|g|\, \big\|_\lt \ \forall\,f,g\in V,
\end{equation}
and have constructed the first examples of such infinite-dimensional subspaces. 
In this situation, if $|f|$ is known then $f$ is uniquely determined almost everywhere
up to an unavoidably arbitrary global phase factor of $\pm 1$;
if $|f|$ is known within a small tolerance in norm
then up to such a global phase factor, $f$ is determined  within a correspondingly small tolerance.
This issue arises for instance in crystallography, 
where one seeks to recover an unknown function $F\in\lt(\reals)$ from the absolute value 
of its Fourier transform $\widehat{F}$. 
Upon substituting $f = \widehat{F}$ and $g = \widehat{G}$,
then invoking Plancherel's theorem
to express $\norm{F\pm G}_\lt = \norm{\widehat{F}\pm\widehat{G}}_\lt = \norm{f\pm g}_\lt$
and $\big\||\widehat{F}|-|\widehat{G}|\big\|_\lt = \big\||f|-|g|\big\|_\lt$,
the inequality \eqref{SPR} expresses a desirable stability in the recovery
of $F$ from $|\widehat{F}|$.

There is an extensive literature concerning phase retrieval, that is,
determination of $f$ from $|f|$ up to unavoidable ambiguity, 
with an emphasis on finite-dimensional subspaces. 
The first result on uniform stability for phase retrieval was achieved by Cand\`es, Strohmer, and Voroninski \cite{CSV}, who used iid random vectors with uniform distribution on the sphere to produce $n$-dimensional subspaces of $m$-dimensional $\ell^2$-spaces satisfying uniformly stable phase retrieval with   $m$  on the order of $n\log(n)$. This was then improved to $m$ being on the order of $n$ in \cite{CL}.  Phase recovery 
for infinite-dimensional subspaces has been shown to be unstable in general by
Cahill, Casazza, and Daubechies \cite{CCD}
and by Alaifari and Grohs \cite{AG}.
We refer to Grohs et.\ al.\ \cite{grohs_K_R}
for an expository article on phase recovery,
and to Calderbank et.\ al.\ \cite{calderbank_etal} for an introduction to the
specific topic of stability for infinite-dimensional subspaces.
The present note develops simple 
examples of infinite-dimensional subspaces in which versions of stable phase retrieval hold. 
These examples include certain variants of Rademacher series and lacunary Fourier series.

For complex-valued functions, the natural quantity on the left-hand
side of the inequality \eqref{SPR} becomes $\min_{|z|=1}\norm{f-zg}_\lt$,
with the minimum taken over all complex numbers $z$ of modulus $1$.
Following Calderbank et.~al. \cite{calderbank_etal}, 
we say that a subspace $V$ of a complex $L_2$-space satisfies stable phase retrieval if
there exists $C<\infty$ such that
\begin{equation} \label{SPR_complex}
\min_{|z|=1}\norm{f-zg}_\lt \le C\big\|\,|f|-|g|\, \big\|_\lt \ \forall\,f,g\in V.
\end{equation}
We generalize the stable phase retrieval inequality in the following way.

\begin{definition}
Let $p\in [1,\infty]$ and let $V$ be 
a subset
of the complex Banach space $L^p(\mu)$ for some measure $\mu$. 
We say that $V$ satisfies $L^p$-H\"older-stable phase retrieval
if there exist parameters $\gamma\in(0,1]$ and $C<\infty$ such that 
\begin{equation} \label{gSPR}
\min_{|z|=1} \norm{f-zg}_{L^p} \le C\big\|\, |f|-|g|\, \big\|_{L^p}^\gamma \cdot
(\norm{f}_{L^p} + \norm{g}_{L^p})^{1-\gamma} \ \ \forall\,f,g\in V.
\end{equation}
We say that $V$ satisfies $L^p$-stable phase retrieval if \eqref{gSPR} holds with $\gamma=1$.
\end{definition}

Stable phase retrieval in the sense \eqref{SPR_complex} is thus 
$L^p$-stable phase retrieval for $p=2$.
The notion of H\"older-stable phase retrieval for subsets has appeared in
work of Cahill, Casazza, and Daubechies \cite{CCD}.
We are primarily interested in subspaces $V$, but in Example~\ref{nonlinear example} below,
$V$ is not a subspace.

We will abbreviate, writing $L^p$-H\"older-SPR and $L^p$-SPR,
and occasionally writing $L^p$-Lipschitz-SPR as a synonym for $L^p$-SPR.
For real Hilbert spaces $\lt(\mu,\reals)$, this definition is modified 
by replacing $\{z\in\complex: |z|=1\}$ by $\{\pm1\}$. We will write
``real $L^p$-SPR''. Only the exponents $p=2,4$ arise in the examples below. 

By defining the equivalence relation $\sim$ on a subspace $V$ by $f\sim g$ if and only if $f=zg$ for some unimodular scalar $z$, we see that $\min_{|z|=1}\|f-zg\|_{L^p}$ is exactly the distance between $f$ and $g$ in the quotient space $V/\sim$. In particular, $V$ satisfies $L^p$-SPR with constant $C$ if and only if the recovery map of $f\in V/\sim$ from $|f|$ is well-defined and $C$-Lipschitz.

Some of our proofs only directly establish
$\lt$-H\"older-SPR with certain specific exponents $\gamma<1$, rather than the formally stronger
property of $\lt$-Lipschitz-SPR.  However, the second and third authors together with Freeman and Oikhberg  have proved 
 \cite[Corollary 3.12]{pineau-taylor}  for both the real and the complex cases
that for any exponent $p\in[1,\infty]$,
for subspaces $V$, $L^p$-H\"older-SPR implies $L^p$-Lipschitz-SPR. 
We will exploit this general result to upgrade conclusions from $\lt$-H\"older-SPR to
$\lt$-Lipschitz-SPR.

\medskip
Let $\mu$ be a probability measure. 
Consider an orthonormal subset $\{r_j: j\in\naturals\}$ 
of the complex Hilbert space $\lt = \lt(\mu) = \lt(\mu,\complex)$.
Let $V\subset\lt$ be the closure of the span of $\{r_j\}$ over $\complex$.
Let $\one$ be the function $\one(x)\equiv 1$.
Define associated functions \begin{equation} s_j = |r_j|^2-\one. \end{equation}

In the case of $\lt(\mu,\complex)$, we consider closed subspaces spanned by
orthogonal sets $\{r_j: j\in\naturals\}$ satisfying the following three hypotheses:
\begin{align}
&
\big\{\one,\, s_i,\, r_j\overline{r_k}: i,j,k\in\naturals \text{ and } j\ne k\big\}
\text{ is an orthogonal set.}
	\label{H1}
\\& \sup_j \norm{r_j}_{L^4}<\infty. \label{H2}
\\&
\text{There exists $\delta>0$ such that }
\inf_i \norm{r_i}_4^4 \ge 1+\delta \text{ and }  
\inf_{j\ne k} \norm{r_j\overline{r_k}}_2^2 \ge \delta. 
	\label{H3}
\end{align}

Since $\norm{s_i}_2^2 = \norm{r_i}_4^4-2\norm{r_i}_2^2+1 = \norm{r_i}_4^4-1$
by the hypothesis that $\norm{r_i}_2=1$,
the first part of hypothesis \eqref{H3} can be equivalently restated
as $\norm{s_i}_2^2\ge\delta>0$.

A consequence of these hypotheses is that
$V\subset L^4$ and there exists $C<\infty$ such that
\begin{equation} \label{L4bound} 
\norm{f}_{L^4}\le C\norm{f}_{L^2}\ \forall\,f\in V.  \end{equation}
Indeed,
if $f = \sum_k a_k r_k$ with $(a_k: k\in\naturals)\in\ell^2$ then
$|f|^2$ is represented as the pairwise orthogonal sum
\begin{equation} \label{expansion} |f|^2 = \sum_{i\ne j} a_i\overline{a_j}r_i\overline{r_j}
+ \sum_k |a_k|^2 s_k  + \norm{f}_2^2\cdot\one.\end{equation}
The $L^4$ norm bound follows using orthogonality and the Cauchy-Schwarz inequality, 
since $\norm{r_i\overline{r_j}}_2\le \norm{r_i}_4\norm{r_j}_4$
and $\norm{s_k}_2 \le 1+\norm{r_k^2}_2\le 1+\norm{r_k}_4^2$
are uniformly bounded by \eqref{H2}.
The inequality \eqref{L4bound}, and a similar $L^6$ norm inequality that
holds under stronger hypotheses, are pillars of our reasoning.

Let $\{r_j\}\subset\lt(\mu,\complex)$ be an orthonormal set of complex-valued functions
satisfying hypotheses \eqref{H1},\eqref{H2},\eqref{H3}, and let $V$ be as above. 
We begin by observing that 
$|f|$ determines $f$ uniquely, up to multiplication  by a unimodular complex scalar,
for each $f\in V$. Indeed,
$|f|$ certainly determines $f$ if $|f|=0$ almost everywhere.
Consider next any $0\ne f\in V$. Expand $f = \sum_k a_k r_k$, 
with $a\in\ell^2$. Then $|f|^2\in\lt$, and has expansion \eqref{expansion}.
The terms of this sum are mutually orthogonal, and the series converges in $\lt$ norm.
Therefore $|f|^2$ determines each of the coefficients in this expansion;
it determines each $|a_n|^2$ and each  product $a_i\overline{a_j}$.
Choose some $n_0$ satisfying $a_{n_0}\ne 0$.
Writing $a_n = |a_n|e^{i\arg(a_n)}$,
$\arg(a_n)-\arg(a_{n_0})$ is determined 
modulo $2\pi\integers$
by $|a_n|^2$, $|a_{n_0}|^2$, and $a_n\overline{a_{n_0}}$.
Therefore $|f|^2$ and $\arg(a_{n_0})$ together determine all coefficients $a_n$,
and hence determine $f$, up to multiplication by $z = e^{i\arg(a_{n_0})}$.

Note that this reconstruction of $f$ from $|f|$ is not stable in the sense
desired, since it requires division by $|a_{n_0}|$, for which no {\em a priori}
positive lower bound is available. Note also that it exploits
only the coefficients of $s_k$ and of $r_n\overline r_{n_0}$.

The next result asserts that under these same hypotheses, the reconstruction
of $f$ from $|f|$ can be done stably.

\begin{proposition} \label{prop1}
Let $\mu$ be a probability measure.
Let $\{r_j\}\subset\lt(\mu,\complex)$ be an orthonormal set of complex-valued functions
satisfying hypotheses \eqref{H1},\eqref{H2},\eqref{H3}. 
Then $V$ satisfies $L^4$-SPR. 
\end{proposition}

Under a supplementary hypothesis, 
Proposition~\ref{prop1} has an almost immediate implication for $\lt$-stable phase retrieval.

\begin{corollary} \label{cor:L6}
Let $\{r_n\}$ satisfy the hypotheses of Proposition~\ref{prop1}.
Assume also that there exist $q>4$ and $C<\infty$ such that
$V\subset L^q(\mu)$ and
\begin{equation} \label{RH} \norm{f}_{L^q} \le C\norm{f}_{L^2} \ \forall\, f\in V.\end{equation}
Then $V$ satisfies $\lt$-stable phase retrieval.
\end{corollary}

Proposition~\ref{prop1} and Corollary~\ref{cor:L6} will be proved below.

As is well known,
for any even integer $q\ge 6$, the inequality \eqref{RH} holds whenever
the functions $r_j$ are independent 
random variables, have uniformly bounded $L^q$ norms, and satisfy $r_j\perp\one$.
Indeed, consider the case $q=6$. If $\norm{r_n}_6\le A<\infty$ for all $n$  then
\begin{align*}
\norm{\sum_n a_n r_n}_6^6
	&= \sum_{i_1,i_2,i_3}
 \sum_{j_1,j_2,j_3}
\prod_{k=1}^3 a_{i_k}\,
\prod_{l=1}^3 \overline{a_{j_l}}\,\,
	\Big\langle r_{i_1}r_{i_2}r_{i_3},\,
	{r_{j_1}} {r_{j_2}} {r_{j_3}}\Big\rangle
\\& \le 
\sum_n |a_n|^6 A^6
	+ \binom{6}{2}A^6\sum_m\sum_n |a_m|^4 |a_n|^2
	+ \binom{6}{3} A^6\sum_m\sum_n |a_m|^3 |a_n|^3
\end{align*}
since 
$\Big\langle r_{i_1}r_{i_2}r_{i_3},\,
{r_{j_1}} {r_{j_2}} {r_{j_3}}\Big\rangle =0$
unless each of the six indices that appear in the inner product, appears at least twice.
The same reasoning applies for arbitrary even integers $q\ge 8$. 

We next present a class of examples 
based on Proposition~\ref{prop1} and Corollary~\ref{cor:L6}.
The construction involves sums of independent random variables, and may
be contrasted with a more elaborate construction in \cite{calderbank_etal},
which combines independent summands with summands having pairwise disjoint supports.

\begin{example} \label{example1}
Let $\mu$ be a probability measure.
Let $r_n$ be independent identically distributed complex-valued random variables in $L^6(\mu)$
satisfying $\norm{r_n}_{\lt}=1$. 
Assume that 
\begin{align} 
&r_n\perp\one  \text{ and } r_n^2\perp \one
\\ &\mu(\{x: |r_n(x)|\ne 1\})>0.  
\label{no_rademacher}
\end{align} 
Then $\{r_n\}$ satisfies the hypotheses of Proposition~\ref{prop1}, 
and satisfies those  of Corollary~\ref{cor:L6} with $q=6$.
Therefore the closure of its span in $L^2(\mu)$  satisfies both $L^4$-SPR and $L^2$-SPR.
\end{example}

Example~\ref{example1} does not apply to Rademacher series,
for which $r_n = \pm1$ each with probability $\tfrac12$, 
violating hypothesis \eqref{no_rademacher}.
Nor do Rademacher series satisfy phase retrieval,
since $|r_m|\equiv |r_n|$ for all $m,n$.

In the formulation of Example~\ref{example1},
the hypothesis $r_n^2\perp \one$, together with independence, 
ensure that $r_i\overline{r_j}\perp r_j\overline{r_i}$ whenever $i\ne j$, since
\[ \langle r_i\overline{r_j},\,r_j\overline{r_i}\rangle 
= \int r_i^2\, \overline{r_j}^2\,d\mu
= \int r_i^2\,d\mu\cdot \overline{\int r_j^2\,d\mu}
= \langle r_i^2,\one\rangle \cdot \overline{\langle r_j^2,\one\rangle} =0.\]
The hypothesis that $|r_n|$ is not equal almost everywhere to $1$
ensures that $\norm{s_n}_2\ne 0$.
The other hypotheses of Proposition~\ref{prop1}, and the 
embedding of $V$ into $L^6$, are consequences
of independence, identical distribution, and the assumption that $r_n\perp\one$.
Details of the verifications are left to the reader.

\medskip
Before indicating other classes of examples with stable phase retrieval,
we prove Corollary~\ref{cor:L6} and Proposition~\ref{prop1}.

\begin{proof}[Proof of Corollary~\ref{cor:L6}]
By H\"older's inequality,
\[ \norm{|f|-|g|}_4
\le \norm{|f|-|g|}_2^{\theta} (\norm{f}_q + \norm{g}_q)^{1-\theta}
\le C^{1-\theta}\norm{|f|-|g|}_2^{\theta} (\norm{f}_2 + \norm{g}_2)^{1-\theta}\]
where $\theta\in (0,1)$ is defined by the relation
$\tfrac14 = \tfrac{\theta}2 + \tfrac{1-\theta}{q}$.
Therefore for any $f,g\in V$, by H\"older's inequality and Proposition~\ref{prop1},
\begin{align*}
\min_{|z|=1} \norm{f-zg}_2
		\le
\min_{|z|=1} \norm{f-zg}_4
\le C' \norm{|f|-|g|}_4
\le C''\norm{|f|-|g|}_2^{\theta} (\norm{f}_2 + \norm{g}_2)^{1-\theta}.
	\end{align*}
Thus $L^2$-H\"older SPR holds. $L^2$-Lipschitz SPR follows from \cite[Corollary 3.12]{pineau-taylor}.
\end{proof}

The proof of Proposition~\ref{prop1} relies on the following elementary inequality.

\begin{lemma} \label{lemma:lemma}
Let $\{r_j\}$ satisfy hypotheses \eqref{H1}, \eqref{H2}, and \eqref{H3}.
For any $f,g\in V$,
\begin{equation} \label{lemma3_inequality}
\norm{|f|^2-|g|^2}_2^2 \ge 
\delta \Big[ \norm{f}_2^2 \norm{g}_2^2 - |\langle f,g\rangle|^2 \Big]
+ (\norm{f}_2^2-\norm{g}_2^2)^2.
\end{equation}
\end{lemma}
We prove Proposition~\ref{prop1}  assuming Lemma~\ref{lemma:lemma}, and then prove Lemma~\ref{lemma:lemma} below.

\begin{proof}[Proof of Proposition~\ref{prop1}]
By multiplying by scalars and interchanging the roles of $f,g$
if necessary, we may assume with no loss of generality that $\norm{f}_2\le\norm{g}_2=1$.
By Cauchy-Schwarz, 
\begin{equation} \label{inequality}
    \begin{split}
         \norm{|f|^2-|g|^2}_2 
\le \big\|\,|f|+|g|\,\big\|_4 \cdot \big\|\,|f|-|g|\,\big\|_4
&\le C(\norm{f}_2 + \norm{g}_2)\,\big\|\,|f|-|g|\,\big\|_4
\\
&\le 2C\,\big\|\,|f|-|g|\,\big\|_4.
    \end{split}
\end{equation}
Write $f = re^{i\theta}g+h$ with $r\ge 0$, $\theta\in\reals$, and $h\perp g$.
Then $|\langle f,g\rangle|^2 = r^2$
and
\begin{equation} \label{another_identity} \norm{f}_2^2\norm{g}_2^2 - |\langle f,g\rangle|^2
= (r^2+\norm{h}_2^2)- r^2 = \norm{h}_2^2.\end{equation}
Let $\delta \in(0,1]$ be a parameter for which
the conclusion	\eqref{lemma3_inequality} of Lemma~\ref{lemma:lemma} holds. Inserting \eqref{another_identity} into \eqref{lemma3_inequality} gives
\begin{equation*} 
	\delta \norm{h}_2^2 +	(1-r^2-\norm{h}_2^2)^2
\le \norm{|f|^2-|g|^2}_2^2 \le 4C^2\norm{|f|-|g|}_4^2. \end{equation*}
Therefore since $0<\delta\le 1$,
\begin{equation*}
\delta \norm{h}_2^2 +	\tfrac14\delta (1-r^2-\norm{h}_2^2)^2
\le 4C^2\norm{|f|-|g|}_4^2. \end{equation*}
The left-hand side is 
\[ (\delta - \tfrac12\delta(1-r^2)) \norm{h}_2^2 + \tfrac14\delta (1-r^2)^2 + \tfrac14\delta\norm{h}_2^4
\ge \tfrac12\delta\norm{h}_2^2 + \tfrac14\delta(1-r^2)^2\]
and therefore since $(1-r)\le (1-r^2)$,
\[ \norm{h}_2^2 + (1-r)^2 \le 16 C^2 \delta^{-1} \norm{|f|-|g|}_4^2.\] 
Defining $z = e^{i\theta}$,
$\norm{f-zg}_2^2 = \norm{h}_2^2 + (1-r)^2$ and therefore
\[ \norm{f-zg}_2^2 \le 16 C^2 \delta^{-1}  \norm{|f|-|g|}_4^2.\]
Since $f-zg\in V$, its $L^4$ norm is majorized by a constant multiple of
its $\lt$ norm. Thus
$\norm{f-zg}_4\le C'\norm{|f|-|g|}_4$ for another finite constant $C'$ 
which depends on $\delta$.
\end{proof}

\begin{proof}[Proof of Lemma~\ref{lemma:lemma}]
Under the hypothesis that $\norm{r_j}_2=1$, 
$\norm{s_j}_2^2 = \norm{r_j}_4^4-1$.
Therefore the hypothesis $\inf_j \norm{r_j}_4^4\ge 1+\delta$
is equivalent to $\inf_j \norm{s_j}_2^2 \ge\delta$.

Express $f,g\in V$ as $f = \sum_k a_k r_k$ and $g = \sum_k b_k r_k$.
By \eqref{expansion},
\begin{align} |f|^2-|g|^2 
\label{lemma3_first}
= \sum_{i\ne j} (a_i \overline{a_j}-b_i\overline{b_j}) r_i\overline{r_j}
	+  (\norm{f}_2^2-\norm{g}_2^2)\one 
+ \sum_{k} (|a_k|^2-|b_k|^2) s_k
\end{align}
where $\one$ is the constant function $1$.
The functions $\one$, $s_k$, and $r_i\overline{r_j}$ with $i\ne j$ are pairwise orthogonal
by hypothesis \eqref{H1}.  Therefore
\begin{align} \norm{|f|^2-|g|^2}_2^2 
&= \sum_{k} ||a_k|^2-|b_k|^2|^2  \norm{s_k}_2^2
+ (\norm{f}_2^2-\norm{g}_2^2)^2
+ \sum_{i\ne j} |a_i \overline{a_j}-b_i\overline{b_j}|^2 \norm{r_i\overline{r_j}}_2^2
\label{lemma3_2nd_display}
\\&
\ge \delta \sum_{k} ||a_k|^2-|b_k|^2|^2  
+ (\norm{f}_2^2-\norm{g}_2^2)^2
+ \delta \sum_{i\ne j} |a_i \overline{a_j}-b_i\overline{b_j}|^2
 \notag
\end{align}
by hypothesis \eqref{H3}.

Algebraic manipulation of the last term on the right-hand side gives
\begin{align*} 
\sum_{i\ne j} |a_i \overline{a_j}-b_i\overline{b_j}|^2
	& = (\sum_k |a_k|^2)^2 
+ (\sum_k |b_k|^2)^2 
-2|\sum_k a_k \overline{b_k}|^2
- \sum_k (|a_k|^2-|b_k|^2)^2.
\\& = \norm{f}_2^4 
+ \norm{g}_2^4 
-2|\langle f,g\rangle|^2
- \sum_k (|a_k|^2-|b_k|^2)^2
\\& = 
2\Big[ \norm{f}_2^2 \norm{g}_2^2 - |\langle f,g\rangle|^2 \Big]
+ (\norm{f}_2^2-\norm{g}_2^2)^2
- \sum_k (|a_k|^2-|b_k|^2)^2.
\end{align*}
Substituting this expression into the preceding lower bound, two terms cancel, leaving 
\begin{align*} \norm{|f|^2-|g|^2}_2^2 
& \ge 2\delta \Big[ \norm{f}_2^2 \norm{g}_2^2 - |\langle f,g\rangle|^2 \Big]
+ (1+\delta) (\norm{f}_2^2-\norm{g}_2^2)^2
\\&
\ge  2\delta\Big[ \norm{f}_2^2 \norm{g}_2^2 - |\langle f,g\rangle|^2 \Big]
+  (\norm{f}_2^2-\norm{g}_2^2)^2.
\end{align*}
\end{proof}

A well-known theme is the analogy between lacunary Fourier series and sums
of independent random variables. Our next two examples express this theme.

\begin{example} \label{example2}
Let $N\ge 2$ and let $P\in\lt([0,1],\complex)$ be a trigonometric polynomial 
\[P(x) = \sum_{k=1}^N \alpha_k e^{2\pi i kx}\]
with coefficients $\alpha_k\in\complex$. 
Suppose that $|P|$ is not constant.
Let $A\in\naturals$ satisfy $A>2N$.  
Let $V\subset \lt([0,1],\complex)$ be the closure of the span of $\{P(A^n x): n\in\naturals\}$.
Then $V$ satisfies both $L^4$-SPR and $\lt$-SPR.
\end{example}

Example~\ref{example2} is an instance of Corollary~\ref{cor:L6}, with arbitrarily large $q<\infty$.
Verification of the hypotheses of the corollary is left to the reader.
The $L^q$ norm inequality \eqref{RH} holds since
$\sum_{n=1}^\infty a_n \sum_{k=1}^N \alpha_k e^{2\pi i A^n kx}$
is a sum of $N$ lacunary Fourier series,
and since any lacunary series with $\ell^2$ coefficients defines
a function in $L^q$ for all $q<\infty$.

The next example is a real analogue of Example~\ref{example2}.

\begin{example} \label{example3}
The closure of the subspace of $\lt([0,1],\reals)$ spanned by $\{\sin(2\pi 4^n x): n\in\naturals\}$
satisfies $L^4$-SPR  and $\lt$-SPR.
\end{example}

Example~\ref{example:rudin}, below, is a more efficient version of Example~\ref{example3}.

If complex rather than real linear combinations are allowed,
then phase retrieval cannot hold in Example~\ref{example3},
nor in any example with two real-valued basis functions $r,r'$.
Indeed, $f=r+ir'$ and $g = \overline{f} =  r-ir'$ satisfy $|f|\equiv |g|$, 
but $f$ is not a constant multiple of $g$.

Proposition~\ref{prop1}  and Corollary~\ref{cor:L6} do not apply to
Example~\ref{example3}, since with $r_n(x) = 2^{1/2}\,\sin(2\pi 4^n x)$
one has $r_i\overline{r_j} = r_j\overline{r_i}$ for all $i,j$.
However, a small modification of the reasoning underlying those two results gives 
Proposition~\ref{prop1B}, whose hypotheses are satisfied in Example~\ref{example3}.  

For Hilbert spaces $\lt(\mu,\reals)$ of real-valued functions with orthonormal bases of real-valued functions $r_n$ we modify the orthogonality hypothesis \eqref{H1} as follows: \begin{equation} \label{H1B} \big\{\one,\, s_i,\, r_j r_k: i,j,k\in\naturals \text{ and } j < k\big\} \text{ is an orthogonal set.} \end{equation} 

\begin{proposition} \label{prop1B}
Let $\mu$ be a probability measure.
Let $\{r_j\}\subset\lt(\mu)$ be an orthonormal set of real-valued functions
satisfying hypotheses \eqref{H2},\eqref{H3},\eqref{H1B}. 
Then the closure $V\subset\lt(\mu,\reals)$ of the span of $\{r_j: j\in\naturals\}$ over $\reals$
satisfies real $L^4$-SPR.

If there exist $q>4$ and $C<\infty$ such that the $L^q$ norm inequality \eqref{RH} holds 
for all functions in $V$
then $V$ satisfies real $\lt$-SPR.
\end{proposition}

The only changes from the proof of Proposition~\ref{prop1} are that 
in \eqref{lemma3_first}, the first term becomes $2\sum_{i<j} (a_ia_j-b_ib_j)r_ir_j$,
and consequently that on the right-hand side of \eqref{lemma3_2nd_display},
the last term is changed to 
\[ 4\sum_{i<j} (a_ia_j-b_ib_j)^2 \norm{r_ir_j}_2^2
= 2 \sum_{i\ne j} |a_ia_j-b_ib_j|^2 \norm{r_i\overline{r_j}}_2^2.
\]
The corresponding quantity in the proof of Proposition~\ref{prop1}
is $ \sum_{i\ne j} |a_ia_j-b_ib_j|^2 \norm{r_i\overline{r_j}}_2^2$.
The new factor of $2$ thus arising is favorable for our purpose.
\qed

\medskip
If $4^n$ is replaced by $3^n$ or $2^n$ in Example~\ref{example3}
then Proposition~\ref{prop1B} no longer applies. Indeed, if $3^n$
is used the desired orthogonality
between $s_n$ and $r_{n+1}r_n$ fails to hold; 
$e^{2\pi i \cdot 2\cdot 3^n x}$
occurs with nonzero coefficient in the
Fourier series for $s_n$, while $e^{2\pi i \cdot 3^{n+1}x}\cdot e^{-2\pi i \cdot 3^n x}
= e^{2\pi i \cdot 2\cdot 3^n x}$ also
occurs with nonzero coefficient in the Fourier series for $r_{n+1}r_n$. 
A similar issue arises for $2^n$.

Another application of Proposition~\ref{prop1B} is a real analogue of Example~\ref{example1}.

\begin{example} \label{example_real1}
Let $\mu$ be a probability measure.
Let $q>4$ be an even integer.
Let $r_n$ be independent identically distributed real-valued random variables in $L^q(\mu)$
satisfying $\norm{r_n}_{\lt}=1$. 
Assume that 
\begin{equation} \left\{ \begin{aligned}
	&r_n\perp\one  
	\\ &\mu(\{x: |r_n(x)|\ne 1\})>0.  
\end{aligned} \right. \end{equation}
Then $\{r_n\}$ satisfies
the hypotheses of Proposition~\ref{prop1B},
and consequently the closure of its span in $L^2(\mu,\reals)$ 
satisfies real $L^4$-SPR and real $\lt$-SPR.
\end{example}

\medskip
We proceed by lightly modifying a construction of
Rudin \cite{rudin} to create examples of trigonometric series related
to the theory of $\Lambda(p)$ sets
that satisfy stable phase retrieval, yet are rather far from being lacunary in nature.
To simplify matters, we set this example in 
the ambient Hilbert space $L^2([0,1]\times[0,1],\complex)$, with respect to
two-dimensional Lebesgue measure, rather than in $L^2([0,1],\complex)$. Define $r_\nu$ to be
\begin{equation} \label{rudin-type} 
r_\nu(x,y) = 2^{1/2} \sin(2\pi \nu y)\,e^{2\pi i n_\nu x}, 
\end{equation}
where $(n_\nu: \nu\in\naturals)$ is a subsequence of $\naturals$ 
to be specified.

To quantify the asymptotic density of a subsequence $(n_\nu)$ of $\naturals$, define
$\alpha(N)$ to be the number of indices $\nu$  satisfying $n_\nu\le N$.

\begin{example} \label{example:rudin}
There exists a strictly increasing sequence $(n_\nu: \nu\in\naturals)$, satisfying 
the asymptotic density lower bound
$\limsup_{N\to\infty} N^{-1/2}\alpha(N)>0$
such that the closed subspace $V$ of $\lt([0,1]\times[0,1])$ 
spanned by the functions $r_\nu$ defined in \eqref{rudin-type}
satisfies $L^4$-SPR.

There exists such a sequence satisfying 
$\limsup_{N\to\infty} N^{-1/3}\alpha(N)>0$ such that $V$ also satisfies $\lt$-SPR. 
\end{example}

Thus these sequences $(n_\nu)$ are far denser than lacunary sequences.

\begin{proof}
In \S4.7 of \cite{rudin}, Rudin constructs a sequence $n_\nu$ that satisfies
$\limsup N^{-1/2}\alpha(N)>0$
such that $n_i+n_j = n_k+n_l$ if and only if $(i,j)$ is a permutation of $(k,l)$,
and deduces from this property the inequality
$\norm{f}_4\le C\norm{f}_2$ for all $\lt$ functions of the form
$f(x) = \sum_\nu c_\nu e^{2\pi i n_\nu x}$. 
Let $(n_\nu)$ be any such sequence, and define $\{r_\nu\}$ by \eqref{rudin-type}.
Hypothesis \eqref{H2}, the uniform upper bound for $\norm{r_\nu}_4$, certainly holds.
The nonconstant factors $\sin(2\pi \nu y)$ 
ensure a uniform lower bound $\norm{r_\nu}_4^4\ge 1+\delta$,
	so \eqref{H3} holds.

To verify hypothesis \eqref{H1}, first consider any inner product
$\langle r_{j}\overline{r_k},\,r_l\overline{r_m}\rangle$
with $j\ne k$ and $l\ne m$.
Calculation of this inner product produces a factor of
$\int_0^1 e^{2\pi i (n_j-n_k-n_l+n_m)x}\,dx$, which vanishes
unless $n_j-n_k-n_l+n_m=0$. Equivalently, 
$n_j+n_m = n_l+n_k$. Therefore by Rudin's construction, $(l,k)$ is a permutation
of $(j,m)$. If $j\ne k$, this implies that $(j,k) = (l,m)$. 
The associated functions $s_{k}(x,y) = 2\sin^2(2\pi ky)-1 = - \cos(4\pi ky)$
are independent of $x$, hence satisfy $s_k\perp r_i\overline{r_j}$ whenever $i\ne j$.
Finally, if $k\ne l$ the  $s_k\perp s_l$
since $\cos(4\pi ky)\perp \cos(4\pi ly)$ in $L^2([0,1])$.

Rudin \cite{rudin} likewise constructs a sequence satisfying $\limsup N^{-1/3}\alpha(N)>0$,
satisfying the same conditions in the preceding paragraph, and satisfying 
$\norm{\sum_\nu b_\nu e^{2\pi i n_\nu x}}_6\le C\norm{b}_{\ell^2}$
for all coefficient sequences $b\in\ell^2$.  Consequently
for any function $f(x,y)$ of the form 
$\sum_\nu a_\nu \sin(2\pi\nu y)e^{2\pi i n_\nu x}$ with $a\in\ell^2$,
\begin{align*} \int_{[0,1]}\int_{[0,1]}\big|\sum_\nu a_\nu \sin(2\pi\nu y)e^{2\pi i n_\nu x}\big|^6\,dx\,dy
& \le C \int_{[0,1]} (\sum_\nu |a_\nu\sin(2\pi\nu y)|^2)^{6/2}\,dy
\\ & 
\le C\int_{[0,1]} (\sum_\nu |a_\nu|^2)^{3}\,dy
=C \norm{a}_{\ell ^2}^6 \leq 8 C \norm{f}_{\lt}^6.  \end{align*}

In each of these two situations, $V$ has the indicated properties.
\end{proof}

\medskip \noindent {\bf Remark.}\ 
In this example, the subspace $V$ is in a sense larger, relative to
other ambient subspaces naturally associated to it,
than is the case for corresponding examples involving lacunary series.
To formulate this assertion more precisely,
for each degree $D\in\naturals$ 
let $V_{N,D}$ be the subspace of $\lt$ spanned by polynomials of degrees $\le D$ in 
$\{r_\nu: 1\le \nu\le N\}$. 
Let $N$ tend to infinity, while $D$ remains fixed.
The dimensions $\dim(V_{N,D})$ satisfy
$\liminf_{N\to\infty} N^{-3}\dim (V_{N,D})<\infty$ for any $D$ 
in Example~\ref{example:rudin}, 
while for the lacunary series example $r_\nu = 2^{1/2}\sin(2\pi 4^\nu x)$,
$\dim(V_{N,D})$ has order of magnitude $N^D$.
Thus the span of $\{r_\nu: 1\le \nu\le N\}$, for these $N$, is a comparatively
large subspace of the associated spaces $V_{N,D}$ in Example~\ref{example:rudin}.
\\

We conclude by giving an example of a subset that  satisfies H\"older-stable phase retrieval
and is invariant under multiplication by unimodular scalars, but is not a subspace.
The aforementioned theorem of Freeman, Oikhberg, Pineau and Taylor \cite{pineau-taylor}  applies only to subspaces,
so we are unable to upgrade the conclusion from H\"older-SPR to Lipschitz-SPR.

 \begin{example}\label{nonlinear example}
Let $\Lambda\subset \mathbb{Z}$, and let $E$ be the set of all $f\in L^2([0,1],\mathbb{C})$ such that $\widehat{f}$ is supported on $\Lambda.$ Suppose that $\Lambda$ has the property that if $n_j\in \Lambda$ and $n_1-n_2=n_3-n_4$ then either $n_1=n_2$ or $n_1=n_3.$  Fix $c=(c_n)_{n\in\Lambda}\in \ell^2(\Lambda)_+$ and define $$E_c=\{f\in E : |\widehat{f}|=c\}=\Big\{\sum_{n\in \Lambda}\gamma_nc_ne^{2\pi inx} : \gamma_n\in \mathbb{C} \ \text{and}\ \ |\gamma_n|=1 \ \forall n\in \Lambda\Big\}.$$
Then $E_c$, equipped with the $L^4$ norm, satisfies \eqref{gSPR} with $\gamma=1$. Moreover,  if for some $q>4$ all $f\in E$ satisfy the $L^q$ bound $\|f\|_{q}\leq C'_{\Lambda}\|f\|_{2}$ then $E_c$ also satisfies   \eqref{gSPR} with $p=2$ and $\gamma = \frac{q-4}{2q-4}$. 
 \end{example}
 \begin{proof}
   We begin by noting that $E\subset L^4([0,1],\mathbb{C})$ and $\|f\|_4\leq C_\Lambda \|f\|_2$ for all $f\in E$. To prove our claim that $E_c$ satisfies \eqref{gSPR} with $p=4$ and $\gamma=1$,  notice that we may assume without loss of generality that $\|c\|_{\ell^2}=1$. In this case, $\|f\|_2=\|g\|_2=1$, and
$$\||f|^2-|g|^2\|_2\leq \||f|-|g|\|_4\||f|+|g|\|_4\leq 2C_{\Lambda}\||f|-|g|\|_4.$$
  We claim that the following identity holds for $f,g\in E$:
 \begin{multline}\label{alg identity for SPR}
     \left|\|f\|_2^2-\|g\|_2^2\right|^2+\left[\|f\|_2^4+\|g\|_2^4-2|\langle f,g\rangle|^2\right]
	\\ =\||f|^2-|g|^2\|_2^2+\sum_{m\in \Lambda}\left(|\widehat{f}(m)|^2-|\widehat{g}(m)|^2\right)^2.
 \end{multline}
 The identity \eqref{alg identity for SPR} implies that $E_c$ satisfies  $L^4$-Lipschitz-stable phase retrieval. Indeed, the second term on the right-hand side of \eqref{alg identity for SPR} vanishes. Since $f,g\in E_c$, they have equal $L^2$ norms, implying that the first term on the left-hand side vanishes. Write $f=re^{i\theta}g+h$, with $0\leq r\leq 1$, $\theta\in \mathbb{R}$, and $h\perp g$. Then, $\|f\|_2^2\|g\|_2^2-|\langle f,g\rangle|^2=\|h\|_2^2=1-r^2.$ To finish the proof,  note that $\|f-e^{i\theta}g\|_2^2=\|h\|_2^2+(1-r)^2\leq 2\|h\|_2^2$,  use the inequality $\|f\|_4\leq C_{\Lambda}\|f\|_2$, and combine the above inequalities.

The derivation of \eqref{alg identity for SPR} is similar to the proof of Lemma~\ref{lemma:lemma}, but easier. The details are left to the reader.  That the supplementary $L^q$ bound  implies that $E_c$ satisfies   \eqref{gSPR} with $p=2$ and $\gamma = \frac{q-4}{2q-4}$ follows from an invocation of H\"older's inequality similar to 
 the one in the proof of Corollary~\ref{cor:L6}. 
 \end{proof}

 \noindent {\bf Remark.}\  The subspace $E$ in Example~\ref{nonlinear example} will not satisfy phase retrieval unless $\Lambda$ has cardinality at most one, as if $m,n\in \Lambda$ and $f=e^{2\pi i nx}$ and $g=e^{2\pi imx}$ then $|f|\equiv|g|$. Observe that on the Fourier side, $f,g$ are disjoint unit vectors in $\ell^2$ when $m\neq n$. Subsets of the form $E_c$ have an opposite  behavior on the Fourier side and appear in the study of random Fourier series.

\bigskip The authors thank Professors Calderbank, Daubechies, Freeman, and Freeman
for generously sharing their work in advance of its formal release.

\end{document}